\documentclass{proc-l}

\usepackage{amscd, amssymb, enumerate, float, mathtools,amsthm}

\usepackage[colorlinks]{hyperref}
\usepackage[capitalize,nameinlink,noabbrev]{cleveref}

\newtheorem{theorem}{Theorem}[section]

\newtheorem{proposition}[theorem]{Proposition}

\theoremstyle{definition}
\newtheorem{definition}[theorem]{Definition}
\newtheorem{example}[theorem]{Example}

\newtheorem{question}[theorem]{Question}
\newtheorem{remark}[theorem]{Remark}

\newcommand{\R}{\mathbb{R}}
\newcommand{\N}{\mathbb{N}}
\newcommand{\Z}{\mathbb{Z}}

\newcommand{\ux}{\underline{x}}
\newcommand{\ua}{\underline{\alpha}}
\newcommand{\ub}{\underline{\beta}}

\newcommand{\ug}{\underline{\gamma}}

\usepackage{amsfonts}

\begin{document}

\markboth{Lewis}
{Factorization in The Monoid of Integrally Closed Ideals}

%

\title{Factorization in The Monoid of Integrally Closed Ideals}
\author{Emmy Lewis}
\curraddr{Department of Mathematics, Cornell University, 310 Malott Hall, Ithaca, NY 14853, USA}
\email{jl2826@cornell.edu}



\maketitle

\begin{abstract}
Given a Noetherian ring $A$, the collection of all integrally closed ideals in $A$ which contain a nonzerodivisor, denoted $ic(A)$, forms a cancellative monoid under the operation $I*J=\overline{IJ}$, the integral closure of the product. The monoid is torsion-free and atomic -- every integrally closed ideal in $A$ containing a nonzerodivisor can be factored in this $*$-product into $*$-irreducible integrally closed ideals. Restricting to the case where $A$ is a polynomial ring and the ideals in question are monomial, we show that there is a surjective homomorphism from the Integral Polytope Group onto the Grothendieck group of integrally closed monomial ideals under translation invariance of their Newton Polyhedra. Notably, the Integral Polytope Group, the Grothendieck group of polytopes with integer vertices under Minkowski addition and translation invariance, has an explicit basis, allowing for explicit factoring in the monoid.
\end{abstract}



\section{Introduction}

We investigate the factorization of integrally closed ideals in commutative rings. Factorization in this manner stem from Zariski's Theorem on the unique factorization in two-dimensional regular local rings \cite{zar13,zar38}. From an example of Lipman and Huneke \cite{lip88} uniqueness of factorization failed in higher dimensions of regular local rings and set-up the study of the monoid of integrally closed ideals. By monoid we mean a semigroup with identity, we now define the monoid in more generality than \cite{lip88}:

\begin{definition}
\label{icAdef}
For a commutative Noetherian ring $A$, define
\[
ic(A)=\{\,I \mid I\, \text{is an integrally closed ideal in}\, A\, \text{and}\, I \,\text{contains a nonzerodivisor} \}
\] 
and put an operation $*$ on $ic(A)$ by setting $I*J=\overline{IJ}$ for $I,J\in ic(A)$. 
\end{definition}

As we see in \cref{cmonthm}, $ic(A)$ is a commutative cancellative monoid; and as we see in \cref{torsionF} $ic(A)$ is a torsion-free monoid. Here by cancellative we mean that $I*J=I*K$ implies $J=K$ for $I,J,K\in ic(A)$. This is the monoid where we focus the study of factorization and note the the case where $A$ is a polynomial ring and the ideals are taken to be monomial will also be of primary interest and factorization of integrally closed monomial ideals has been of interest particularly in 2 and 3 variables by Crispin Qui\~nonez \cite{qui10} and Gately \cite{gat00}, respectively. In the case of arithmetic and factorization of ideals under normal multiplication, recent work has been done in \cite{ger19,ger21,bas21}. Using the general factorization theory of monoids we deduce that factorization can take place in our subject monoid (c.f. \cref{factor}):

\begin{theorem}
For any Noetherian ring $A$, the monoid $ic(A)$ is an atomic monoid. That is, all integrally closed ideals containing nonzerodivisors can be $*$-factored into $*$-irreducible ideals.
\end{theorem}

We refer the reader to \cite{gat05,hei14} for more study on $*$-irreducible ideals, particularly in regular local rings. By $*$-irreducible ideals we mean those ideals $I$ such that $I=\overline{JK}$ implies $J=(1)$ or $K=(1)$. The factorization may not be unique, although it can be unique as in the case of Zariski's Theorem, and more generally can be bounded when the ring is a regular local ring (c.f. \cref{rlrbfm}).

We can imbue $ic(A)$ with a congruence relation that preserves cancellativity, setting up the connection with convex geometry. In particular, when we restrict the relation and monoid to the monomial case in a polynomial ring, the congruence is exactly translation invariance of the Newton Polyhedra. In a more general case, the congruence is via isomorphism of ideals if the ring was normal.

\begin{theorem}
(c.f. \cref{equivMon}) Define a relation $\sim$ on $ic(A)$ by $I\sim J$ if and only if there are nonzero $f,g\in A$ such that $\overline{fI}=\overline{gJ}$ for $I,J\in ic(A)$. Then $\sim$ is a congruence and $ic(A)/\mathord\sim$ is a cancellative monoid.
\end{theorem}

In our study, we introduce the unique method of embedding this monoid where factorization takes place into its Grothendieck group. Recall that commutative cancellative monoids can naturally embed into an Abelian group called its Grothendieck group. If the monoid is written multiplicatively, its Grothendieck group is the group of formal fractions of elements of the monoid, and cancellativity guarantees that the natural map is injective.

Restricting to the case where $ic(R)$ consists of integrally closed monomial ideals in $R=k[x_1,\dots,x_d]$ for some field $k$, and restricting the congruence to monomial elements allows us to make the connection to convex geometry. We call the resulting Grothendieck group $\widetilde{g}(R)$ Specifically we make a connection with the integral polytope group (see \cite[Section 2]{funke}). Consider the monoid of all bounded closed convex sets with vertices lying in $\Z^d$ under Minkowski addition (that is $A+B=\{a+b\,|\,a\in A, b\in B\}$), denoted $P(\Z)$. The monoid is cancellative, so we can embed it inside of its Grothendieck group, denoted $\mathcal{P}(\Z)$. We now place the equivalence relation induced by setting polytopes that are translates of each other as equivalent. With this we obtain the integral polytope group $\mathcal{P}(\Z)_t$.

\begin{theorem}
(c.f. \cref{surjectiveGroup}) There is a surjective group homomorphism $\phi:\mathcal{P}(\Z)_t\to\widetilde{g}(R)$ induced by the map 
\[
\phi([P])=[(\ux^{\ua}\,|\,\ua\in(P_{\N}+\R_+^d)\cap\N^d)]
\]
\end{theorem}

In particular, $\mathcal{P}(\Z)_t$ is a free Abelian group where the basis can be made geometrically explicit as in \cite{funke}. Using that the integral polytope group has an explicit basis, we can use the map $\phi$ to construct a colon ideal factorization of each integrally closed monomial ideal in terms of this basis (c.f. \cref{basisFactor}). It is, then, a primary question to explicitly describe the basis of the integral polytope in terms of the map $\phi$ in order to express integrally closed monomial ideals in terms of this basis.




\section{Preliminary}

We now outline preliminaries of integral closure that we use throughout the article. For full details and proofs we refer the reader to \cite[Chapters 1]{HSInt} and \cite{HubSwa08}.

For an ideal $I$ in a commutative Noetherian ring $A$, we define the \textit{integral closure} of $I$ to be the ideal
\[
\overline{I}\!=\!\{x\!\in\! A\,|\,x^n+a_1x^{n-1}+\dots+a_{n-1}x+a_n\!=\!0 \textnormal{ for some }n\!\in\!\N \textnormal{, } a_i\!\in\! I^i \textnormal{ for }1\!\leqslant\! i\!\leqslant\! n\}.
\]

For monomial ideals, let $R\!=\!\mathbb{K}[x_1,\dots,x_d]$ for a field $\mathbb{K}$ and $d\!\in\!\N$, and $I$ a monomial ideal of $R$. We denote a monomial of $R$ by $\ux^{\ua}\!=\!x_1^{\alpha_1}\dots x_d^{\alpha_d}$ where $\ua\!=\!(\alpha_a,\dotsm,\alpha_d)\!\in\!\N^d$. Then we say the \textit{exponent set} of $I$ is $E(I)\!=\!\{\ua\!\in\!\N^d\,|\,\ux^{\ua}\!\in\! I\}$, and
\[
\textnormal{NP}(I)\!=\!\textnormal{conv}(E(I))
\]
is called the \textit{Newton Polyhedron}, i.e. the convex hull of the exponent set of $I$. The Newton Polyhedron connects convex geometry to the integral closure via \cite[Proposition 1.4.6]{HSInt} so that for any monomial ideal $I$ we have
\[
E(\overline{I})\!=\!\textnormal{NP}(I)\cap\N^d.
\]
That is, we can read off the integral closure of an ideal via the lattice points of the Newton Polyhedron.


\begin{theorem}
\label{cmonthm}
For a commutative Noetherian ring $A$, the pair $(ic(A),*)$ is a commutative cancellative monoid.
\end{theorem}
\begin{proof}
Note that the product of ideals which contain nonzerodivisors still contain nonzerodivisors. Thus $ic(A)$ is closed under $*$ and that the operation is commutative since $A$ is commutative. The ideal $(1)$ is integrally closed so that $I*(1)=(1)*I=\overline{I(1)}=\overline{I}=I$ for $I\in ic(A)$. Hence $ic(A)$ has an identity. By \cite[Exercise 1.1]{HSInt}, in particular because $A$ is Noetherian, for $I,J,K\in ic(A)$ we have
\[
(I*J)*K=\overline{IJ}*K=\overline{\overline{IJ}K}=\overline{IJK}=\overline{I\overline{JK}}=I*(J*K)
\]
so that $ic(A)$ is associative and hence a monoid. By \cite[Lemma 1.14]{lip88} or \cite[Exercise 1.2]{HSInt}, since each ideal contains nonzerodivisors, $ic(A)$ is also cancellative.
\end{proof}

We note that the condition of the ring being Noetherian and the ideals containing nonzerodivisors are necessary conditions for associativity and cancellativity, respectively.

\section{Monoidal Properties}

Here we briefly outline the common additional properties of $ic(A)$ for any ring $A$ that aid in characterizing the monoid.

\begin{proposition}\label{torsionF}
For any ring $A$, the monoid $ic(A)$ has the following properties:
\begin{enumerate}
    \item[$(i)$] $ic(A)$ is \textit{conical}, that is if $I*J=(1)$ then $I=J=(1)$;
    \item[$(ii)$] $ic(A)$ is \textit{reduced}, i.e. the only invertible element is $(1)$, the identity;
    \item[$(iii)$] $ic(A)$ is torsion-free.
\end{enumerate}
\end{proposition}
\begin{proof}
For (\textit{i}) suppose that $I*J=(1)$, that is $\overline{IJ}=(1)$. Then $\overline{IJ}\subseteq\sqrt{IJ}=\sqrt{I}\cap \sqrt{J}=(1)$ so that $1\in \sqrt{I}$ and $1\in \sqrt{J}$. Hence $1\in I$ and $1\in J$ so that $I=J=(1)$. Clearly then for (\textit{ii}), to show $ic(A)$ is reduced, suppose that $I\in ic(A)$ has inverse $J$, then $I*J=(1)$; however, as $ic(A)$ is conical, $I=J=(1)$. Finally, for (\textit{iii}), let $I,J\in ic(A)$ such that $\overline{I^n}=\overline{J^n}$ for some natural number $n$. By \cite[Lemma 2.1]{mca89}, as the ideals contain nonzerodivisors and since $n/n=1/1$ we have that $I=J$.
\end{proof}

Another common characterizing property of monoids is whether the monoid is primary. In particular, in \cite{geroldinger21}, the primary property has importance in understanding some factorization properties. However, as we will see, the monoid of integrally closed ideals being primary is a rare property. 

\begin{definition}
We call a monoid $M$ \textit{primary} if for every $a,b\in M$ we have that there exists a $n\in\mathbb{N}$ such that $b^n\in aM$.
\end{definition}

We note the following result that is analogous to the result that a domain is one-dimensional and local if and only if its multiplicative monoid of nonzero elements is primary.

\begin{proposition}
Suppose $A$ is a domain. The monoid $ic(A)$ being primary implies that $A$ is a dimension $1$ local domain.
\end{proposition}
\begin{proof}
If we take a maximal ideal $\mathfrak{m}$ and some ideal $I\in ic(A)$ then we have that $\overline{\mathfrak{m}^n}\in I*ic(A)$ for some $n\in\mathbb{N}$. That is, there is a $J\in ic(A)$ such that
\[
\mathfrak{m}^n\subseteq\overline{\mathfrak{m}^n}=\overline{IJ}\subseteq I
\]
and conversely there is a $m\in\mathbb{N}$ and a $K\in ic(A)$ with $\overline{I^m}\in\mathfrak{m}*ic(A)$, or as above
\[
\overline{I^m}=\overline{\mathfrak{m}K}\subseteq\mathfrak{m}.
\]
Hence taking radicals we get that $\sqrt{I}\subseteq\mathfrak{m}$. So we have $\mathfrak{m}^n\subseteq I\subseteq\mathfrak{m}$ so that every $I\in ic(A)$ is $\mathfrak{m}$-primary. In particular, since all nonzero primes are in $ic(A)$ since $A$ is a domain, we have that if $\mathfrak{p}$ is a nonzero prime, then $\mathfrak{p}=\mathfrak{m}$. Hence $A$ is a dimension one local domain.
\end{proof}

We now begin our focus on factorization through investigating the atomic property of monoids. Being the weakest form of factorization, the property is a prerequisite for further study of the factorization structure of the monoid of integrally closed ideals.

\begin{definition}
In a monoid $(M,+)$, $a\in M$ is an \textit{atom} if $a=x+y$ implies $x=0$ or $y=0$, where $0$ is the identity. The monoid $(M,+)$ is called \textit{atomic} if $M$ can be generated by atoms.
\end{definition}

In the context of $ic(A)$ we call the atoms in the monoid by the name $*$-irreducible ideals following the literature as in \cite{lip88} and definition of irreducible ideals.

\begin{example}
Let $(R,\mathfrak{m})$ be a two-dimensional regular local ring. Then as $R$ is a domain, \cref{cmonthm} holds. Then by Zariski's theorem (see \cite[Theorems 14.4.4, 14.4.6]{HSInt}), $ic(R)$ is an atomic monoid whose atoms are the irreducible ideals. If $I\in ic(R)$ is irreducible, $I=J*K=\overline{JK}=JK$ implies $J=(1)$ or $K=(1)$. The monoid $ic(R)$ is also said to be \textit{free} on the the subset $S$ of $ic(R)$ of irreducible ideals. 

\end{example}

To show that the monoid of integrally closed ideals is atomic, we must introduce the concept of monoid-ideals, i.e. subsets of monoids with the absorption property familiar to ring-ideals.

\begin{definition}
An subset $W$ of a monoid $(M,+)$ is called an \textit{$M$-ideal} if $W+M\subseteq W$. We call $W$ \textit{principal} if $W=m+M$ for some $m\in M$.
\end{definition}

\begin{theorem}\label{factor}
For any Noetherian ring $A$, the monoid $ic(A)$ is an atomic monoid. That is, all integrally closed ideals containing nonzerodivisors can be $*$-factored into $*$-irreducible ideals.
\end{theorem}
\begin{proof}
By \cite[Proposition 1.1.4]{ger06book}, if a monoid satisfies the ascending chain condition for principal $M$-ideals, then $M$ is atomic. Let 
\[
I_1*ic(A)\subseteq I_2*ic(A)\subseteq\dots\subseteq I_n*ic(A)\subseteq\dots
\]
be an ascending chain of principal $ic(A)$-ideals. Note that $I_i\in I_{i+1}*ic(A)$ for all $i$. Hence, for each $i$ there is a $K_i\in ic(A)$ such that $I_i=\overline{I_{i+1}K_i}$. Then as $I_{i+1}K_i\subseteq I_{i+1}$ we have $I_i=\overline{I_{i+1}K_i}\subseteq\overline{I_{i+1}}=I_{i+1}$ for each $i$. Hence, we have an ascending chain of the ideals 
\[
I_1\subseteq I_2\subseteq\dots\subseteq I_n\subseteq\dots
\]
which stabilizes at say $n$ since $A$ is Noetherian. That is, $I_n=I_{n+1}=\dots$ hence we have
\[
I_n*ic(A)=I_{n+1}*ic(A)=\dots,
\]
i.e. the chain in $ic(A)$ stabilizes. Therefore, $ic(A)$ has the ascending chain property for principal $ic(A)$-ideals, and is thus atomic.
\end{proof}

Now that we know that factorization into $*$-irreducible ideals occurs for integrally closed ideals, we note that prime ideals form part of this base of atoms generating the monoid.

\begin{proposition}
If $\mathfrak{p}\in ic(A)$ is a prime ideal, then $\mathfrak{p}$ is an atom, i.e. $\mathfrak{p}$ is $*$-irreducible.
\end{proposition}
\begin{proof}
Suppose that $\mathfrak{p}=\overline{IJ}$ for $I,J\in ic(A)$. Then $IJ\subseteq\mathfrak{p}$ and without loss of generality, $I\subseteq\mathfrak{p}$. Now $\mathfrak{p}=\overline{IJ}\subseteq I$ and hence we have $\mathfrak{p}=I$ forcing $J=(1)$ so that $\mathfrak{p}$ is $*$-irreducible.
\end{proof}

We note, however, that prime ideals may not be prime elements in the monoid, counter to the case of the monoid of ideals under usual ideal multiplication (c.f. \autoref{lipmanHochster}).

One general question in characterizing factorization in monoids is to investigate which class of factorization the monoid belongs. In particular, whether the factorization is unique as in the case of Zariski's theorem, whether there are finitely many factorizations of any given elements (called a finite factorization monoid), or whether the lengths of all factorizations of a given element are bounded. Uniqueness and finiteness are clearly stronger properties than boundedness, and we will be focusing on boundedness here.

\begin{definition}
An atomic monoid $M$ is called a \textit{bounded factorization monoid} or \textit{BFM} if for any element $m\in M$ the set \[
L(m)=\{ k\,|\,m=m_1m_2\dots m_k \text{ for atoms } m_i \}
\]
is finite; i.e. the length of any factorization is bounded.
\end{definition}

\begin{theorem}\label{rlrbfm}
If $(A,\mathfrak{m})$ is a regular local ring, then $ic(A)$ is a bounded factorization monoid (BFM).
\end{theorem}
\begin{proof}
Recall that for any element $x\in A$ we can set
\[
\text{ord}_{\mathfrak{m}}(x) = \text{max}\{ n \mid x\in\mathfrak{m}^n\}
\]
and that since $A$ is a regular local ring we have that this is a discrete valuation of rank one. So we can define 
\[
\text{ord}_{\mathfrak{m}}(I) = \text{max}\{ n \mid I\subseteq\mathfrak{m}^n\}
\]
for any ideal $I\subseteq A$. As $A$ is a domain and $\text{ord}_{\mathfrak{m}}$ is a valuation we have that $\text{ord}_{\mathfrak{m}}(\overline{I})=\text{ord}_{\mathfrak{m}}(I)$ \cite[Theorem 6.7.8]{HSInt}. Since $A$ is a regular local ring, $\text{ord}_{\mathfrak{m}}(IJ)=\text{ord}_{\mathfrak{m}}(I)+\text{ord}_{\mathfrak{m}}(J)$ for ideals $I,J\subseteq A$. Hence we have a monoid homomorphism $\text{ord}_{\mathfrak{m}}:ic(A)\rightarrow\mathbb{N}$ since
\[
\text{ord}_{\mathfrak{m}}(\overline{IJ})=\text{ord}_{\mathfrak{m}}(IJ)=\text{ord}_{\mathfrak{m}}(I)+\text{ord}_{\mathfrak{m}}(J).
\]
In particular, since we can $*$-factor any ideal $I\in ic(A)$ into $*$-irreducible ideals, say $I_1,\dots I_k$ we have that $\text{ord}_{\mathfrak{m}}(I)=\text{ord}_{\mathfrak{m}}(I_1)+\dots+\text{ord}_{\mathfrak{m}}(I_k)$.

Since the only ideal with order $0$ is $A$, our identity here, the length of any factorization is bounded by $\text{ord}_{\mathfrak{m}}(I)$ since the longest factorization we can have here is where $\text{ord}_{\mathfrak{m}}(I_i)=1$ for all $i$. 
\end{proof}

\begin{remark}
Notice in the above proof that if $\text{ord}_{\mathfrak{m}}(I)=1$, then $I$ is $*$-irreducible.
\end{remark}

\begin{example}\label{lipmanHochster}
Not all $*$-irreducible ideals are prime ideals, and the factorization into $*$-irreducible ideals isn't necessarily unique. Let $A=k[[x,y,z]]$ for some field $k$. Then from \cite[Equation (0.1)]{lip88} we have
\[
(x,y,z)(x^3,y^3,z^3,xy,xz,yz)=(x^2,y,z)(x,y^2,z)(x,y,z^2)
\]
where each ideal is $*$-irreducible. Note also that $\text{ord}_{\mathfrak{m}}((x^3,y^3,z^3,xy,xz,yz))=2$, $\text{ord}_{\mathfrak{m}}((x^2,y,z))=1$, $\text{ord}_{\mathfrak{m}}((x,y^2,z))=1$, and $\text{ord}_{\mathfrak{m}}((x,y,z^2))=1$.

This example also shows how prime ideals may not be prime elements. We have that the maximal ideal, $\mathfrak{m} = (x,y,z)$ divides the right-hand side product, call these ideals $J_1$, $J_2$, and $J_3$, respectively. Then, assuming $\mathfrak{m}$ is prime in the monoid, without loss of generality we examine two cases, $\mathfrak{m}$ divides $J_1*J_2$ and $\mathfrak{m}$ divides $J_3$. For these cases, say $J_1*J_2=\mathfrak{m}*K$ or $J_3=\mathfrak{m}*K'$ for some $K,K'\in ic(A)$. Thus, we have that
\[
\mathfrak{m}*K*J_3 = \mathfrak{m}*(x^3,y^3,z^3,xy,xz,yz)\,\,\textnormal{ or }\,\,J_1*J_2*\mathfrak{m}*K' = \mathfrak{m}*(x^3,y^3,z^3,xy,xz,yz)
\]
Then we can cancel $\mathfrak{m}$ on both sides:
\[
K*J_3 = (x^3,y^3,z^3,xy,xz,yz)\,\,\textnormal{ or }\,\,J_1*J_2*K' = (x^3,y^3,z^3,xy,xz,yz)
\]
yielding a proper factorization of a $*$-irreducible ideal, namely $(x^3,y^3,z^3,xy,xz,yz)$, a contradiction. To verify that this is a proper factoring, note that $K$ is not the identity since $J_3\neq (x^3,y^3,z^3,xy,xz,yz)$. Hence $\mathfrak{m}$ cannot be a prime element in the monoid $ic(A)$.
\end{example}

We also note here that the above example shows that, in general, $ic(A)$ may not be a half-factorial monoid. Here, half-factorial means that the lengths of all factorizations of a given element are equal.

\section{Operations on $A$, Submonoids, and the Grothendieck Group}

Transitioning from the internal properties of $ic(A)$ to investigate the factorization, we introduce how operations on the ring behave, how we can incorporate graded rings, and how we can begin to investigate the Grothendieck group of the monoid of integrally closed ideals. 

\begin{remark}
Let $S\subseteq A$ be a multiplicative closed set. Then 
\[
ic(S^{-1}A)= \{\, S^{-1}I \mid I\in ic(A), I\cap S=\emptyset \}\
\]
since all ideals in $S^{-1}A$ are extended, and $I\cap S=\emptyset$ preserves the ideal containing a nonzerodivisor. The operation $*$ is compatible here since for $I,J\in ic(A)$ we have
\[
\overline{(S^{-1}I)(S^{-1}J)}=\overline{S^{-1}IJ}=S^{-1}\overline{IJ}.
\]
\end{remark}

\begin{proposition}
Let $B$ be a faithfully flat extension of $A$, then $ic(A)$ embeds in $ic(B)$
\end{proposition}
\begin{proof}
For each $I\in ic(A)$ we'll consider $\overline{IB}\in ic(B)$. Then as $B$ is a faithfully flat extension, we still have that $\overline{IB}$ contains a nonzerodivisor. We also have 
\[
\overline{IJB}=\overline{(IB)(JB)}=\overline{(IB)}*\overline{(JB)}.
\]
It remains to check that $\overline{IB}=\overline{JB}$ implies $I=J$. Indeed, as $B$ is a faithfully flat extension
\[
J=\overline{JB}\cap A=\overline{IB}\cap A =\overline{I}=I
\]
finishing the proof.
\end{proof}

\begin{definition}
If $A$ is a $\mathbb{N}^d\times\mathbb{Z}^e$ graded ring, then as the product of homogeneous ideals is homogeneous, and in this setting so is the integral closure, we can recognize the subset of homogeneous ideals in $ic(A)$ as a submonoid of $ic(A)$. Call this submonoid $ic(A)_{hom}$
\end{definition}
\begin{remark}
The theorems above about $ic(A)$ hold for $ic(A)_{hom}$ too. Hence $ic(A)_{hom}$ is a commutative, torsion-free, cancellative, reduced, atomic monoid. That is, the factorization into $*$-irreducible ideals now becomes a factorization into homogeneous $*$-irreducible ideals.
\end{remark}

\begin{definition}
Let $R=k[x_1,\dots,x_d]$ or $k[[x_1,\dots,x_d]]$ for some field $k$. Then similarly to before we can define $ic(R)_{mon}$ to be the subset of monomial ideals in $ic(R)$. 
\end{definition}
\begin{remark}
Since monomial ideals are homogeneous ideals in the $\mathbb{N}^n$ grading of $R$, $ic(R)_{mon}=ic(R)_{hom}$. So we can $*$-factor integrally closed monomial ideals into integrally closed $*$-irreducible monomial ideals.
\end{remark}

\begin{remark}
Using the Grothendieck construction, cancellative monoids can be naturally embedded into a group, called the Grothendieck group. For $ic(A)$ we will call this group $g(ic(A))$. The elements of the Grothendieck group are equivalence classes of pairs $[I,J]\in g(ic(A))$ where $I,J\in ic(A)$ and $[I,J]=[I',J']$ when $I*J'=I'*J$ (i.e. $\overline{IJ'}=\overline{I'J}$). The operation on $g(ic(A))$ is the obvious $[I,J]*[I',J']=[I*I',J*J']$ which makes the inverse of $[I,J]$ the pair $[J,I]$, and the identity is clearly $[(1),(1)]$. We will write these pairs as formal fractions (or formal differences if the monoid was written additively).
\end{remark}

\begin{definition}
For commutative monoids $M,N$ with monoid map $f:M\rightarrow N$, define the map on the Grothendieck groups $G(f):G(M)\rightarrow G(N)$ to be $G(f)(\frac{m_1}{m_2})=\frac{f(m_1)}{f(m_2)}$. 
\end{definition}

From this, it is easy to see that $G(-)$ is a functor from the category of commutative monoids to the category of Abelian groups. Indeed, when the monoids are cancellative, the Grothendieck construction is a functor which is left exact and faithful since cancellative monoids embed in their Grothendieck groups.

\begin{proposition}
\label{pcAprop}
The set $pc(A)=\{\,\overline{(f)} \mid f\in A\, \,\text{and}\, f\, \text{is a nonzerodivisor}\}$ is a submonoid of $ic(A)$
\end{proposition}
\begin{proof}
Clearly $(1)\in pc(A)$. Let nonzero $f,g\in A$, then 
\[
\overline{(f)}*\overline{(g)}=\overline{(f)(g)}=\overline{(fg)}\in pc(A).
\]
\end{proof}

\begin{remark}
Under the Grothendieck construction on $pc(A)$, we have that $G(pc(A))$ becomes a subgroup of $G(ic(A))$ under a similar argument to \cref{pcAprop}. We'll call this subgroup $G(pc(A))$, and denote $\widetilde{g}(A)=G(ic(A))/G(pc(A))$.
\end{remark}

\begin{theorem}\label{equivMon}
Define a relation $\sim$ on $ic(A)$ by $I\sim J$ if and only if there are nonzero $f,g\in A$ such that $\overline{fI}=\overline{gJ}$ for $I,J\in ic(A)$. Then $\sim$ is a congruence and $ic(A)/\mathord\sim$ is a cancellative monoid.
\end{theorem}
\begin{proof}
We'll first show that $\sim$ is an equivalence relation. Clearly $\sim$ is reflexive and symmetric. For $I,J,K\in ic(A)$ suppose $I\sim J$ and $J\sim K$. In this case suppose $\overline{fI}=\overline{gJ}$ and $\overline{hJ}=\overline{lK}$ for nonzero $f,g,h,l\in A$. Then by \cite[Exercise 1.1]{HSInt}
\[
\overline{hfI}=\overline{h\overline{fI}}=\overline{h\overline{gJ}}=\overline{hgJ}=\overline{ghJ}=\overline{g\overline{hJ}}=\overline{g\overline{lK}}=\overline{glK}
\]
so that as $A$ is a domain $gl$ and $hf$ are nonzero, we have $I\sim K$. So we now must show that $\sim$ is a congruence, i.e. if $I\sim J$ and $K\sim H$ then $I*K\sim J*H$ for $I,J,K,H\in ic(A)$. Suppose $\overline{fI}=\overline{gJ}$ and $\overline{hK}=\overline{lH}$ for nonzero $f,g,h,l\in A$. Then by \cite[Exercise 1.1]{HSInt} again
\[
\overline{fh\overline{IK}}=\overline{fhIK}=\overline{fI\overline{HK}}=\overline{gJ\overline{lH}}=\overline{glJH}=\overline{gl\overline{JH}}
\]
so that $I*K\sim J*H$. Hence $\sim$ is a congruence. Thus, $ic(A)/\mathord\sim$ is a monoid and it remains to show that it is cancellative.
Denote the equivalence class of $I\in ic(A)$ in $ic(A)/\sim$ by $[I]$. Suppose that $[I]*[K]=[J]*[K]$ for $I,J,K\in ic(A)$, that is to say $[\overline{IK}]=[\overline{JK}]$, i.e. $\overline{IK}\sim \overline{JK}$. So there are nonzero $f,g\in A$ with $\overline{f\overline{IK}}=\overline{g\overline{JK}}$ and by \cite[Exercise 1.1]{HSInt} again we have
\[
\overline{fI}*K=\overline{fI\overline{K}}=\overline{f\overline{IK}}=\overline{g\overline{JK}}=\overline{gJ\overline{K}}=\overline{gJ}*K
\]
so that as $ic(A)$ is cancellative we have $\overline{fI}=\overline{gJ}$, that is $I\sim J$ or $[I]=[J]$. Hence $ic(A)/\mathord\sim$ is cancellative.
\end{proof}

\begin{definition}
With the above congruence, we'll call $ic(A)/\mathord\sim = \widetilde{ic}(A)$.
\end{definition}

\begin{remark}
If $A$ is normal, then $\overline{fI}=f\overline{I}$ for nonzero $f\in A$. Hence $I\sim J$ if and only if $fI=gJ$, that is, if and only if $I\cong J$ as $A$-modules. So $ic(A)/\mathord\sim$ becomes isomorphism classes of integrally closed ideals.
\end{remark}


\begin{remark}
If $A$ is a (semi) local normal domain then by a theorem of Sally \cite{sal75} normal ideals (i.e. those ideals with $\overline{I^n}=I$ for all $n\in\N$) can't be weak-torsion in $\widetilde{ic}(A)$ since for positive height ideals $\mu(I^n)=1$ implies $\mu(I)=1$, i.e. that $I$ is principal.
\end{remark}

\begin{proposition}
\label{gpisothm}
As groups we have $\widetilde{g}(A)\cong G(\widetilde{ic}(A))$, where $G(\widetilde{ic}(A))$ is the Grothendieck group of $\widetilde{ic}(A)$.
\end{proposition}
\begin{proof}
Define $\phi:\widetilde{g}(A)\to G(\widetilde{ic}(A))$ by $[I,J]*pc(A)\mapsto [[I],[J]]$ for $I,J\in ic(A)$. Let $[I,J]*pc(A)=[K,H]*pc(A)$, that is $[I,J]*[K,H]^{-1}=[\overline{IH},\overline{JK}]\in pc(A)$ which furthermore means that there are nonzero $f,g\in A$ such that $[\overline{IH},\overline{JK}]=[\overline{(f)},\overline{(g)}]$, i.e. $\overline{gIH}=\overline{fJK}$. But this implies that $[\overline{IH}]=[\overline{JK}]$, that is $[I]*[H]=[K]*[J]$ so that $[[I],[J]]=[[K],[H]]$. Hence, $\phi$ is well-defined.

We have that 
\begin{equation}
\begin{split}
\phi ([I,J]*pc(A))*\phi ([K,H]*pc(A)) \\ 
=[[I],[J]]*[[K],[H]] \\ 
=[[\overline{IK}],[\overline{JH}]] \\ 
=\phi ([\overline{IK},\overline{JH}]*pc(A)) \\ 
=\phi ((([I,J]*pc(A))*([K,H]*pc(A))) \\
\end{split}
\end{equation}
for $I,J,K,H\in ic(A)$ so that $\phi$ is a group homomorphism. It's clear that $\phi$ is also onto. Now suppose that $\phi ([I,J]*pc(A))=\phi ([K,H]*pc(A))$, then $[\overline{IH}]=[\overline{KJ}]$ so that for nonzero $f,g\in A$ we have $\overline{gIH}=\overline{fKJ}$ which says that $[\overline{IH},\overline{KJ}]=[\overline{(f)},\overline{(g)}]\in pc(A)$. Hence $[I,J]*pc(A)=[K,H]*pc(A)$ so that $\phi$ is injective, and thus an isomorphism.
\end{proof}

As we will see in the following section, we can restrict $ic(A)$ and $pc(A)$ to submonoids such as integrally closed monomial ideals and monomials, respectively. With this restriction, we can use the same proof above to understand the resulting group. 

\begin{remark}
A cancellative monoid $M$ is torsion-free if and only if the Grothen-dieck group of that monoid, $G(M)$ is torsion-free. Indeed this follows from the fact that $a^n=b^n$ in $M$ if and only if $(\frac{a}{b})^n=1$ in $G(M)$. Thus, by \cref{torsionF}, we have that $G(ic(A))$ is a torsion-free group. We can also see that by \cref{gpisothm}, $\widetilde{g}(A)$ is torsion-free if and only if $\widetilde{ic}(A)$ is torsion-free as well.
\end{remark}

\begin{question}
Do there exist $I,J\in \widetilde{ic}(A)$ such that there are nonzero $f,g\in A$ and $n\in \mathbb{N}$ with $\overline{fI^n}=\overline{gJ^n}$ and $[I]\neq [J]$? In the normal ring setting, are there ideals $I,J\in ic(A)$ such that $\overline{I^n}\cong\overline{J^n}$ for some $n\in\N$ where $I$ is not isomorphic to $J$ as $A$-modules?
\end{question}


\section{Connections to Convex Geometry and The Polytope Group}

We now consider $R=k[x_1,\dots,x_d]$ or $k[[x_1,\dots,x_d]]$ for some field $k$. Here we focus on the submonoid of integrally closed monomial ideals and use the already strong connections with convex geometry via the Newton Polyhedrons.

\begin{definition}
Let $P(\Z)_u$ be the monoid of finitely supported closed convex sets in $\R^d$ with integer vertices under Minkowski addition. Here finitely supported closed convex sets means each convex set is an intersection of finitely many closed half-spaces.
\end{definition}

Here we use the subscript $u$ to denote these polyhedrons may be unbounded, though that is not necessarily the case.

\begin{proposition}\label{newtonHom}
There is a monoid injection $ic(R)_{mon}$ into $P(\Z)_u$ via taking Newton Polyhedrons.
\end{proposition}    
\begin{proof}
For $I\in ic(R)_{mon}$ we note that it is known that $\textnormal{NP}(I)\in P(\Z)_u$. Furthermore, we have $\textnormal{NP}(I*J)=\textnormal{NP}(\overline{IJ})=\textnormal{NP}(IJ)=\textnormal{NP}(I)+\textnormal{NP}(J)$. To see this, as $\textnormal{NP}(I)=\textnormal{conv}(E(I))$ where $E(I)$ denotes the exponent set of $I$, it is enough to show that $E(IJ)=E(I)+E(J)$ since $\textnormal{conv}(S+P)=\textnormal{conv}(S)+\textnormal{conv}(P)$ for any $S,P\subseteq\R^d$ \cite[Theorem 1.1.2]{sch14}. It is clear that $E(I)+E(J)\subseteq E(IJ)$, for the converse, if $\ug\in E(IJ)$ then $\ux^{\ug}\in IJ$ so that (as it is monomial) we have $\ux^{\ug}=\ux^{\ua}\ux^{\ub}$ for some $\ux^{\ua}\in I$ and $\ux^{\ub}\in J$, yielding equality. Hence
\[
\textnormal{NP}(\overline{IJ})=\textnormal{NP}(IJ)=\textnormal{conv}(E(IJ))=\textnormal{conv}(E(I))+\textnormal{conv}(E(J))=\textnormal{NP}(I)+\textnormal{NP}(J)
\]
so that the map $\textnormal{NP}(-)$ is a monoid map. Now, if $\textnormal{NP}(I)=\textnormal{NP}(J)$, then the vertices must be equal, hence the ideals generated by the vertices must be equal. But this ideal is always a reduction of the original ideal, hence $I=\overline{I}=\overline{J}=J$ finishes the injection.
\end{proof}


Restricting to the bounded elements of $P(\Z)_u$ we have the monoid of polytopes under Minkowski addition, which is cancellative, and we can embed into its Grothendieck group.

\begin{definition}
From \cite[Section 2]{funke} let $\mathcal{P}(\Z)_t$ be the \textit{integral polytope group}. That is, take the monoid of all bounded closed convex sets with vertices lying in $\Z^d$ under Minkowski addition, denoted $P(\Z)$. The monoid is cancellative, embed it inside of its Grothendieck group, denoted $\mathcal{P}(\Z)$. Placing the equivalence relation induced by setting polytopes that are translates of each other as equivalent we obtain the group $\mathcal{P}(\Z)_t$
\end{definition}

\begin{remark}
As polytopes are bounded, for each polytope $[P]\in\mathcal{P}(\Z)_t$ where $[P]$ denotes the translation equivalence class, there is a polytope $[P_{\N}]$ with vertices in $\N$ obtained by translating $P$ into the positive orthant so that $[P]=[P_{\N}]$.
\end{remark}

\begin{definition}
Place an equivalence relation $~$ on $ic(R)_{mon}$ via multiplication by monomials, i.e. $I~J$ if and only if there are monomials $\ux^{\ua},\ux^{\ub}\in R$ such that $\ux^{\ua}I=\ux^{\ub}J$. Using the same argument as \cref{equivMon}, $~$ indeed is an equivalence relation, and under this equivalence relation, the monoid $\widetilde{ic}(R)$ is again a cancellative monoid. Denote the Grothendieck group as $\widetilde{g}(R)$, as above.
\end{definition}

\begin{remark}
From \cref{newtonHom} we can see that the equivalence relation by monomial multiplication is an equivalence on the Newton Polyhedra via translation. That is, two ideals in $ic(R)_{mon}$ will be equivalent if and only if their Newton Polyhedra are translates of one another.
\end{remark}

Now that we have accounted for translation invariance in the monoid of integrally closed ideals, we can create a well-defined surjective map from the integral polytope group.

\begin{theorem}\label{surjectiveGroup}
There is a surjective group homomorphism $\phi:\mathcal{P}(\Z)_t\to\widetilde{g}(R)$ induced by the map
\[
\phi([P])=[(\ux^{\ua}\,|\,\ua\in(P_{\N}+\R_+^d)\cap\N^d)]
\]
\end{theorem}

\begin{proof}
We first show that the map is well-defined. It is enough to check well-definedness at the monoid level. Suppose that $P_{\N}$ and $P_{\N}'$ are two translates of $P$ into $\R_+^d$. Then there are $\ua,\ub\in\N^d$ such that $P_{\N}+\ua=P_{\N}'+\ub$. Clearly we have $\ua+P_{\N}+\R_+^d=\ub+P_{\N}'+\R_+^d$ and furthermore
\begin{equation}\label{translate}
    \ua+((P_{\N}+\R_+^d)\cap\N^d)=\ub+((P_{\N}'+\R_+^d)\cap\N^d).
\end{equation}
Denote $A=(P_{\N}+\R_+^d)\cap\N^d$ and $B=(P_{\N}'+\R_+^d)\cap\N^d$, then let $I_A=(\ux^{\underline{\gamma}}\,|\,\underline{\gamma}\in A)$ and $I_B$ similarly. Then, from \cref{translate} we have $\ux^{\ua}I_A=\ux^{\ub}I_B$. Hence, under the equivalence relation of $\widetilde{g}(R)$ we have $[I_A]=[I_B]$, and so $\phi$ is well-defined.

We will now show that $\phi$ is a homomorphism. Clearly the zero polyhedron equivalence class maps to $[R]$ under $\phi$. Now consider polytopes $P, Q$ and consider $A=(P_{\N}+\R_+^d)\cap\N^d$, $B=(Q_{\N}+\R_+^d)\cap\N^d$, and $I_A=(\ux^{\ua}\,|\,\ua\in A)$ and $I_B$ similarly. Then we have that 
\[
\textnormal{NP}(I_A)=\textnormal{conv}(A)=P_{\N}+R_+^d
\]
since $P_{\N}$ is convex. Hence, as we have $\phi([P])=[I_A]$ and $\phi([Q])=[I_B]$ we see that
\[
\textnormal{NP}(\overline{I_AI_B})=\textnormal{NP}(I_A)+\textnormal{NP}(I_B)=P_{\N}+R_+^d+Q_{\N}+R_+^d=(P_{\N}+Q_{\N})+\R_+^d.
\]
Thus 
\begin{equation}
\begin{split}
    \phi([P+Q])=[(\ux^{\ua}\,|\,\ua\in(P+Q)_{\N}+\R_+^d)\cap\N^d]\\
    =[(\ux^{\ua}\,|\,\ua\in(P_{\N}+Q_{\N}+\R_+^d)\cap\N^d]\\
    =[\overline{I_AI_B}]=[I_A]*[I_B]=\phi([P])*\phi([Q])
\end{split}
\end{equation}
where the second equality follows from the translation invariance and the third follows since the exponent set of an integrally closed ideal comes from the lattice points of the Newton polyhedron and $\textnormal{NP}(\overline{I_AI_B})=(P_{\N}+Q_{\N})+\R_+^d$. Hence, inducing $\phi$ on the whole group, we have the desired homomorphism. 

To show that $\phi$ is surjective, let $[I]\in\widetilde{g}(R)$. Note that $\textnormal{NP}(I)=\textnormal{conv}(\textnormal{gens}(I))+\R_+^d$ so that $[I]=\phi([\textnormal{conv}(\textnormal{gens}(I))])$. 
\end{proof}

\begin{remark}
Notice that we can use a basis of $\mathcal{P}(\Z)_t$ to produce a generating set in $\widetilde{g}(R)$ via $\phi$. Our goal is to define a large subgroup of $\widetilde{g}(R)$ that embeds in $\mathcal{P}(\Z)_t$ to produce a factorization in a submonoid of $ic(R)$ unique up to translation. 
\end{remark}

\begin{remark}\label{basisFactor}
Via $\phi$ we can use the basis of $\mathcal{P}(\Z)_t$ for factorization in $\widetilde{g}(R)$ since $\phi$ is onto. Indeed, if $\mathcal{B}$ is a basis consisting of (translations classes) of polytopes (rather than formal differences) in $\mathcal{P}(\Z)_t$, and $I\in ic(R)_{mon}$, then there are $B_1,\dots,B_n\in\mathcal{B}$ and $m_1,\dots,m_n\in\Z$ such that 
\[
[I]=[\overline{\phi(B_1)^{m_1}\dots\phi(B_n)^{m_n}}]
\]
so that if $m_1,\dots,m_n\in\Z$ we can write equality in the group $g(R)$ as
\[
\ux^{\ua}I=\overline{\ux^{\ub}\phi(B_1)^{m_1}\dots\phi(B_n)^{m_n}}
\]
for some $\ua,\ub\in\N^d$. After simplifying the possibly negative exponents we have, in the monoid (i.e. equality in $R$), that there are $\ua,\ub\in\N_+^d$, $B_1,\dots,B_n,C_1,\dots,C_l\in\mathcal{B}$, and $m_1,\dots,m_n,p_1,\dots,p_l\in\N$ so that
\[
\overline{\ux^{\ua}\phi(C_1)^{p_1}\dots\phi(C_l)^{p_l}I}=\overline{\ux^{\ub}\phi(B_1)^{m_1}\dots\phi(B_n)^{m_n}}.
\]
Or, equivalently using the cancellation property that $\overline{HK}:K=\overline{H}$ (see \cite[Lemma 1.14]{lip88}) we have that
\begin{equation}
\begin{split}
    \overline{\ux^{\ub}\phi(B_1)^{m_1}\dots\phi(B_n)^{m_n}}:\overline{\ux^{\ua}\phi(C_1)^{p_1}\dots\phi(C_l)^{p_l}}\\
    =\overline{\ux^{\ua}\phi(C_1)^{p_1}\dots\phi(C_l)^{p_l}I}:\overline{\ux^{\ua}\phi(C_1)^{p_1}\dots\phi(C_l)^{p_l}}=I.
\end{split}
\end{equation}
Hence we have that 
\[
I = \overline{\ux^{\ub}\phi(B_1)^{m_1}\dots\phi(B_n)^{m_n}}:\overline{\ux^{\ua}\phi(C_1)^{p_1}\dots\phi(C_l)^{p_l}}.
\]
It is our primary goal, then, to describe the elements of $\mathcal{B}$ under the image of $\phi$.
\end{remark}

We now recall some background for the basis constructed in \cite{funke} to begin the study of its image under $\phi$ (c.f. \cite[Definition 3.2]{funke}). 

\begin{definition}
For a polytope $P\subset\R^d$, define its \textit{height} to be $h(P)=\textnormal{inf}\{\alpha_d\,|\,\ua\in P\}$. Let $h\in\R$ and define the projection map onto $h$ in $\R^d$, $c_h$, by $c_h(x_1,\dots,x_d)=(x_1,\dots,x_{d-1},h)$. We define the \textit{shadow} of a polytope to be
\[
Sh(P)=\textnormal{conv}(P\cup c_h(P)).
\]
\end{definition}

From \cite[Lemma 3.3]{funke} we have that $Sh$ is a group homomorphism on $\mathcal{P}(\Z)_t$. Then, following the construction of the basis of the integral polytope group in \cite[Proposition 4.3, Construction 4.4]{funke} we inductively build $\mathcal{B}_1\subseteq\mathcal{B}_2\subseteq\dots\subseteq\mathcal{B}_d=\mathcal{B}$ where $\mathcal{B}_n$\textbackslash$\mathcal{B}_{n-1}$ consists of only polytopes of dimension $n$. We set $\mathcal{B}_1$ to be the translation classes of 1-dimensional polytopes which are not multiples of one another. We then build $\mathcal{B}_n$ from $\mathcal{B}_{n-1}$ by taking the collection of all rank $n$ pure subgroups of $\Z^d$, denoted $\mathcal{U}_n$, then intersecting $\mathcal{B}_{n-1}$ with each element of $U\in\mathcal{U}_n$ and construct the set $\mathcal{C}_n^U=\{Sh(B)\,|\,B\in\mathcal{B}_{n-1}\cap U\textnormal{ and }Sh(B)\textnormal{ is }n\textnormal{-dimensional}\}$ which extends to a basis of $\mathcal{P}(U)_t$ via \cite[Proposition 4.3]{funke}. Taking the union over all of $\mathcal{U}_n$ of these bases yields $\mathcal{B}_n$.

\begin{example}
As in \cite[Example 4.6]{funke} for $d=2$ we would have that the basis of $\mathcal{P}(\Z)_t$ are all the 1-dimensional polytopes which are not multiples of each other, and their shadows, i.e. right triangles. In the context of the monoid of integrally closed monomial ideals in two variables, say $x$ and $y$, we use \cref{basisFactor} to state that each $I\in ic(k[x,y])_{mon}$ can be written as 
\[
I = \overline{x^{a_0}y^{b_0}(x^{a_1},y^{b_1})\dots(x^{a_l},y^{b_l})}:\overline{x^{c_0}y^{d_0}(x^{c_1},y^{d_1})\dots(x^{c_m},y^{d_m})}
\]
for some $a_0,\dots,a_l,b_0,\dots,b_l,c_0,\dots,c_m,d_0,\dots,d_m\in\N$ and applying $\phi$ to 1-dimen-sional polytopes (i.e. lines) and right triangles after we translate the vertices to lay on the $x-$ and $y-$axes which is possible for both polytopes in this context. 
\end{example}


\section*{Acknowledgements}
The author would like to thank New Mexico State University Department of Mathematical Sciences for their support while researching this topic. The author would also like to thank Jonathan Monta{\~n}o for his helpful discussions and support as an advisor at New Mexico State University. The author would also like to thank Alfred Geroldinger for his helpful comments and discussions in generalizing some results.

\bibliographystyle{amsplain}
\bibliography{refs} 

\end{document}